\documentclass[11pt,leqno,oneside,letterpaper]{amsart}
\usepackage[letterpaper,left=1.75truein, top=1truein]{geometry}
\usepackage{amsmath,amsfonts,amssymb,amscd,amsthm,amsbsy,epsf,graphics}

\usepackage{color}

\usepackage[colorlinks,linkcolor=blue,citecolor=blue]{hyperref} 
\usepackage{epsfig} 
\usepackage{graphicx}

\newtheorem{thm}{Theorem}[section]
\newtheorem{cor}[thm]{Corollary}
\newtheorem{lemma}[thm]{Lemma}
\newtheorem{prop}[thm]{Proposition}
\newtheorem{conj}[thm]{Conjecture}

\newtheorem{que}[thm]{Question}

\newtheorem{claim}[thm]{Claim}

\theoremstyle{definition}

\newtheoremstyle{cases}
  {12pt plus 6 pt}
  {2pt}
  {\bfseries}   
  {}
  {\bfseries}
  {.}
  {.5em}
  {}
\theoremstyle{cases}

\numberwithin{subcase}{case} \numberwithin{subsubcase}{subcase}
\numberwithin{equation}{subsection} 

\begin{document}

\title{Graph manifolds $\mathbb Z$-homology $3$-spheres and taut foliations}

\author{Michel Boileau}  
\thanks{Michel Boileau was partially supported by l'Institut Universitaire de France.}
\address{Institut de math\'etiques de Toulouse, UMR 5219 et Institut Universitaire de France, Universit\'e Paul Sabatier 31062 Toulouse Cedex 9, France.}
\email{boileau@math.univ-toulouse.fr }
 
\author[Steven Boyer]{Steven Boyer}
\thanks{Steven Boyer was partially supported by NSERC grant RGPIN 9446-2008}
\address{D\'epartement de Math\'ematiques, Universit\'e du Qu\'ebec \`a Montr\'eal, 201 avenue du Pr\'esident-Kennedy, Montr\'eal, QC H2X 3Y7.}
\email{boyer.steven@uqam.ca}
\urladdr{http://www.cirget.uqam.ca/boyer/boyer.html}

\begin{abstract} 
We show that a graph manifold which is a $\mathbb Z$-homology $3$-sphere not homeomorphic to either $S^3$ or $\Sigma(2,3,5)$ admits a horizontal foliation. This combines with known results to show that the conditions of {\it not} being an L-space, of having a left-orderable fundamental group, and of admitting a co-oriented taut foliation, are equivalent for graph manifold $\mathbb Z$-homology $3$-spheres. 

\end{abstract}

\maketitle

\begin{center}
\today 
\end{center}

Throughout this paper we shall often use $\mathbb Q$-homology $3$-sphere to abbreviate {\it rational homology $3$-sphere} and $\mathbb Z$-homology $3$-sphere to abbreviate {\it integer homology $3$-sphere}. 

Heegaard Floer theory is  a package of  3-manifold homology 
 invariants developed by Ozsv\'ath and Szab\'o \cite{OS3}, \cite{OS2} which provides relatively powerful tools to
distinguish between manifolds. For a rational homology $3$-sphere $M$, the simplest version of these invariants comes in the form of $\mathbb Z/2$-graded abelian groups $\widehat{HF}(M)$ whose Euler characteristic satisfies:   $\chi (\widehat{HF}(M)) = \vert H_{1}(M)\vert$. In particular,  $\hbox{rank} \; \widehat{HF}(M) \geq \vert H_1(M) \vert $.

Ozsv\'ath and Szab\'o defined the family of $L$-{\it spaces} as the class of rational homology $3$-spheres $M$ for which the Heegaard Floer homology is as simple as possible. In other words, $\hbox{rank} \; \widehat{HF}(M ) = \vert H_1(M) \vert$. Examples of L-spaces include  the 3-sphere, lens spaces, and, more generally, manifolds admitting elliptic geometry. By Perelman's proof of the geometrisation conjecture, these are the closed 3-manifolds with finite fundamental group. Beyond these examples, Ozsv\'ath and Szab\'o have shown that the 2-fold branched covering of any non-split alternating link is an L-space, thus providing infinitely many examples of hyperbolic L-spaces. None of these examples are integer homology $3$-spheres, except for $S^3$ and the Poincar\'e sphere $\Sigma(2,3,5)$. 
 
The last decade has shown that the conditions of {\it not} being an L-space, of having a left-orderable fundamental group, and of admitting a $C^2$  co-oriented taut foliation, are strongly correlated for an irreducible $\mathbb Q$-homology $3$-sphere $W$: 
\begin{itemize}

\vspace{0cm} \item the three conditions are equivalent for non-hyperbolic geometric manifolds (cf. \cite{BRW}, \cite{LS}, \cite{BGW}).

\vspace{.3cm} \item Ozsv\'ath and Sz\'abo have shown that if $W$ admits a $C^2$ co-orientable taut foliation then it is not an L-space \cite[Theorem 1.4]{OS1}.

\vspace{.3cm} \item Calegari and Dunfield have shown that the existence of a co-orientable taut foliation on an atoroidal $W$ implies that the commutator subgroup 
$[\pi_1(W), \pi_1(W)]$ is a left-orderable group \cite[Corollary 7.6]{CD}. 

\vspace{.3cm} \item Boyer, Gordon and Watson have conjectured that $W$ has a left-orderable fundamental group if and only if it is not an L-space and have provided supporting evidence in \cite{BGW}.

\vspace{.3cm} \item Lewallen and Levine have shown that strong L-spaces do not have left-orderable fundamental groups \cite{LL}. 

\end{itemize}

Recall that a {\it graph manifold} is a compact, irreducible, orientable $3$-manifold whose Jaco-Shalen-Johannson (JSJ) pieces are Seifert fibred spaces. In this paper we focus on the case that $W$ is an integer homology $3$-sphere, and in particular one which is a graph manifold.  

We begin with the statement of the {\it Heegaard-Floer Poincar\'e conjecture}, due to Ozsv\'ath and Sz\'abo. 

\begin{conj} {\rm (Ozsv\'ath-Sz\'abo)}
An irreducible integer homology $3$-sphere is an L-space if and only if it is either $S^3$ or the Poincar\'e homology $3$-sphere $\Sigma(2,3,5)$. 
\end{conj}
The truth of this striking conjecture would imply that among prime $3$-manifolds, the $3$-sphere is characterized by its Heegaard-Floer homology together with the vanishing of its Casson invariant (or even its $\mu$ invariant). It is known to hold in many instance, for example for integer homology $3$-spheres obtained by surgery on a knot in $S^3$ \cite[Proposition 5]{HW}. 
It lends added interest to the questions: 
\begin{itemize}

\vspace{0cm} \item  Which $\mathbb Z$-homology $3$-spheres admit co-oriented taut foliations?

\vspace{.5cm} \item Which $\mathbb Z$-homology $3$-spheres have left-orderable fundamental groups?

\end{itemize}

We assume throughout this paper that foliations are $C^2$-smooth. The works of Eisenbud-Hirsh-Neumann \cite{EHN}, Jankins-Neumann \cite{JN}
and Naimi \cite{Na}  give necessary and sufficient conditions for a Seifert fibered 3-manifold to carry a horizontal 
foliation. It follows from their work that a Seifert manifold $\mathbb Z$-homology $3$-sphere is an L-space if and only if it is either $S^3$ or the Poincar\'e homology $3$-sphere $\Sigma(2,3,5)$ (cf. Proposition \ref{seifert case}; see also \cite{LS}, \cite{CM}). More recently, Clay, Lidman and Watson have shown that the fundamental group of a graph manifold $\mathbb Z$-homology $3$-sphere is left-orderable if and only if it is neither $S^3$ nor $\Sigma(2,3,5)$ \cite{CLW}. (By convention, the trivial group is {\it not} left-orderable.) The main result of this paper proves Ozsv\'ath-Sz\'abo conjecture for $\mathbb Z$-homology $3$-spheres which are graph manifolds: we show that a graph manifold $\mathbb Z$-homology $3$-sphere admits a co-oriented taut foliation if and only if it is neither $S^3$ nor $\Sigma(2,3,5)$. Before stating the precise version of our result, we need to introduce some definitions.

A {\it transverse loop} to a codimension one foliation $\mathcal{F}$ on a $3$-manifold $M$ is a loop in $M$ which is everywhere transverse to $\mathcal{F}$. 
A codimension one foliation on a $3$-manifold $M$ is {\it taut} if each of its leaves meets a transverse loop. 

A foliation is {\it $\mathbb R$-covered} if the leaf space of the pull-back foliation on the universal cover $\widetilde M$ of $M$ is homeomorphic to the real line. 

A foliation on a $\mathbb Z$-homology $3$-sphere is always co-orientable.

We assume that the pieces of a graph manifold are equipped with a fixed Seifert structure. Note that this structure is unique up to isotopy when the graph manifold is a $\mathbb Z$-homology $3$-sphere (cf. Proposition \ref{horizontal}(2)). 

A surface in a graph manifold $W$ is {\it horizontal} if it is transverse to the Seifert fibres of each piece of $W$. It is {\it rational} if its intersection with each JSJ torus is a union of simple closed curves. A codimension $1$ foliation of $W$ is {\it horizontal}, respectively {\it rational}, if each of its leaves has this property. Horizontal foliations are obviously taut and they are known to be $\mathbb R$-covered \cite[Proposition 7]{Br2}. Rational foliations on graph manifold $\mathbb Z$-homology $3$-spheres are necessarily horizontal (Lemma \ref{strongly rational not vertical}). 
Here is our main result. 

\begin{thm} \label{taut}
Let $W$ be a graph manifold which is a $\mathbb Z$-homology $3$-sphere and suppose that $W$ is neither $S^3$ nor $\Sigma(2,3,5)$. Then $W$ admits a rational foliation. 
\end{thm}

An action of a group $G$ on the circle is called {\it minimal} if each orbit is dense. 

A homomorphism $\rho: G \to \hbox{Homeo}_+(S^1)$ is called {\it minimal} if the associated action on $S^1$ is minimal. 

\begin{cor} \label{consequences} 
Let $W$ be a graph manifold which is a $\mathbb Z$-homology $3$-sphere and suppose that $W$ is neither $S^3$ nor $\Sigma(2,3,5)$. Then

$(1)$ $W$ is not an L-space.

$(2)$ $\pi_1(W)$ admits a minimal homomorphism $\rho$ with values in $\hbox{Homeo}_+(S^1)$ whose image contains a nonabelian free group.

$(3)$ {\rm (Clay-Lidman-Watson \cite{CLW})} $\pi_1(W)$ is left-orderable. 

\end{cor}

\begin{proof}
Since $W$ is a $\mathbb Z$-homology $3$-sphere, the taut foliation $\mathcal{F}$ given by Theorem \ref{taut} is co-orientable. Thus $W$ cannot be an L-space \cite[Theorem 1.4]{OS1}. Assertion $(3)$ is a consequence of the assertion $(2)$; since $H^2(W) \cong \{0\}$, the homomorphism $\pi_1(W) \to \hbox{Homeo}_+(S^1)$ lifts to a homomorphism $\pi_1(W) \to \widetilde{\hbox{Homeo}}_+(S^1) \leq \hbox{Homeo}_+(\mathbb R)$ with non-trivial image. Theorem 1.1(1) of \cite{BRW} now implies that $\pi_1(W)$ is left-orderable. (This also follows from the fact that $\pi_1(W)$ acts non-trivially on $\mathbb R$ by orientation-preserving homeomorphisms since $\mathcal{F}$ is co-oriented and $\mathbb R$-covered \cite[Proposition 7]{Br2}.) Finally, assertion $(2)$ follows from Lemma \ref{minimal} below.
\end{proof}

\begin{lemma}\label{minimal} 
Let $M$ be a $\mathbb Z$-homology $3$-sphere which admits a taut foliation $\mathcal{F}$. Then $\pi_1(M)$ admits a minimal homomorphism $\rho: \pi_1(M) \to \hbox{Homeo}_+(S^1)$ whose image contains a nonabelian free group.
\end{lemma}

\begin{proof}
A theorem of Margulis \cite[Corollary 5.15]{Gh} shows that the image of a minimal representation $\rho: \pi_1(M) \to \hbox{Homeo}_+(S^1)$ is either abelian or  contains a nonabelian free group. The former is not possible since $\pi_1(M)$ is perfect, so to complete the proof we must show that such a representation exists. 

Since $M$ is a $\mathbb Z$-homology $3$-sphere, the co-orientability of $\mathcal{F}$ implies that it has no compact leaves (\cite[Proposition 2.1]{Go}. See also \cite[Part II, Lemma 3.8] {God}). Then by Plante's results \cite[Theorem 6.3, Corollaries 6.4 and 6.5]{Pla}, every leaf of $\mathcal{F}$ has exponential growth, and thus $\mathcal{F}$ admits no non-trivial holonomy-invariant transverse measure. Hence Candel's uniformization theorem \cite[Theorem 12.6.3]{CC1} applies to show that there is a Riemannian metric on $M$ such that $\mathcal{F}$ is leaf-wise hyperbolic. In this setting, Thurston's universal circle construction yields a homomorphism $\rho_{univ}$ of $\pi_1(M)$ with values in  $\hbox{Homeo}_+(S^1)$ \cite{CD}. 

If $L$ denotes the leaf space of the pullback $\widetilde{\mathcal{F}}$ of the foliation $\mathcal{F}$ to the universal cover $\widetilde M$ of $M$, then either $L$ is Hausdorff and $\mathcal{F}$ is $\mathbb{R}$-covered or $L$ has branching points. We treat these cases separately.

First suppose that $\mathcal{F}$ is $\mathbb{R}$-covered. Then Proposition 2.6 of \cite{Fen} implies that after possibly collapsing at most countably many foliated $I$-bundles, we can suppose that $\mathcal{F}$ is a minimal foliation (i.e. each leaf is dense). If $\mathcal{F}$ is ruffled (\cite[Definition 5.2.1]{Ca1}), Lemma 5.2.2 of \cite{Ca1} shows that the associated action of $\pi_1(M)$ on the universal circle of $\mathcal{F}$ is minimal, so we take $\rho = \rho_{univ}$. If $\mathcal{F}$ is not ruffled, it is uniform and so by \cite[Theorem 2.1.7]{Ca1}, after possibly blowing down some pockets of leaves, we can suppose that $\mathcal{F}$ slithers over the circle (\cite[Definition 2.1.6]{Ca1}). Thus if $\widetilde M$ denotes the universal cover of $M$, there is a locally trivial fibration $\widetilde M \to S^1$ whose fibres are unions of leaves of the pull back of $\mathcal{F}$ to $\widetilde M$. Further, the deck transformations of the cover $\widetilde M \to M$ act by bundle maps and so determine a homomorphism of $\pi_1(M)$ with values in $\hbox{Homeo}_+(S^1)$. If this representation has a finite orbit, then a finite index subgroup of $\pi_1(M)$ acts freely and properly discontinuously on a fibre of the fibration $\widetilde M \to S^1$. This is impossible as each fibre is a surface and a finite index subgroup of $\pi_1(M)$ is the fundamental group of a closed $3$-manifold. Therefore
by \cite[Propositions 5.6 and 5.8]{Gh}, the associated action on $S^1$  is semiconjugate to a minimal action $\rho: \pi_1(M) \to \hbox{Homeo}_+(S^1)$. 

In the case that $L$ branches, $\rho_{univ}: \pi_1(M) \to \hbox{Homeo}_+(S^1)$ is faithful. (See the last line of the first paragraph of \cite[\S 6.28]{CD}.) If it branches in both directions, an application of \cite[Lemma 5.5.3]{Ca3} to any finite cover of $M$ implies that $\rho_{univ}(\pi_1(M))$ has no periodic orbit. The conclusion then follows as above from \cite[Propositions 5.7 and 5.8]{Gh}. Thus we are left with the case where $\mathcal{F}$ has one-sided branching, say in the negative direction (cf. \cite{Ca2}). As in the case of $\mathbb{R}$-covered foliations, we can suppose every leaf dense by \cite[Theorem 2.2.7]{Ca2}. We need only show that the action associated to the faithful representation $\rho_{univ}: \pi_1(M) \to \hbox{Homeo}_+(S^1)$ has no finite orbits as otherwise \cite[Theorem 1.2]{Mat} implies that $\rho_{univ}$ is semiconjugate to an abelian representation, which is trivial since $\pi_1(M)$ is perfect. Hence the action of $\rho_{univ}(\pi_1(M))$ on $S^1$ has an uncountable compact set $\Sigma$ of  global fixed points. By \cite[Theorem 3.2.2]{Ca2} the image of $\Sigma$ is dense in almost every circle at infinity of the leaves of  $\widetilde{\mathcal{F}}$, and hence in $S^1_{univ}$ by the construction of the universal circle, see \cite[Theorem 3.4.1]{Ca2}. This contradicts the faithfullness of $\rho_{univ}$. When $M$ is hyperbolic, we can also obtain a contradiction to the existence of a finite orbit from that of topologically pseudo-Anosov elements of $\rho_{univ}(\pi_1(M))$ which have at most finitely many  fixed points in $S^1_{univ}$, see \cite[Lemma 4.2.5]{Ca2}. This completes the proof of the lemma and therefore that of Corollary \ref{consequences}. 
\end{proof}

The conclusion of Lemma \ref{minimal} combines with the two questions above to motivate the following question:

\begin{que}
For which aspherical $\mathbb Z$-homology $3$-spheres $M$ does $\pi_1(M)$ admit a minimal representation to $ \hbox{Homeo}_+(S^1)$?
\end{que}

Our discussion above yields the following corollary. 

\begin{cor} \label{consequences 2} 
The following conditions are equivalent for $W$ a graph manifold $\mathbb Z$-homology $3$-sphere:

$(a)$ $\pi_1(W)$ is left-orderable. 

$(b)$ $W$ is not an L-space.

$(c)$ $W$ admits a rational foliation. 
\qed

\end{cor}

Sections 1 and 2 contain background material on, respectively, the pieces of graph manifold $\mathbb Z$-homology $3$-spheres and strongly detected slopes on the boundaries of Seifert fibered $\mathbb Z$-homology solid tori. Theorem \ref{taut} is proven in \S \ref{existence}.

\section{Pieces of graph manifold $\mathbb Z$-homology $3$-spheres} \label{sec: detecting slopes}

A torus $T$ in a $\mathbb Z$-homology $3$-sphere $W$ splits $W$ into two $\mathbb Z$-homology solid tori $X$ and $Y$. Let $\lambda_X$ and $\lambda_Y$ be primitive classes in $H_1(T)$ which are trivial in $H_1(X)$ and $H_1(Y)$ respectively. The associated slopes on $T$, which we also denote by $\lambda_X$ and $\lambda_Y$, are well-defined. We refer to these slopes as the {\it longitudes} of $X$ and $Y$. A simple homological argument shows that $X(\lambda_Y)$ and $Y(\lambda_X)$ are $\mathbb Z$-homology $3$-spheres while $X(\lambda_X)$ and $Y(\lambda_Y)$ are $\mathbb Z$-homology $S^1 \times S^2$'s. 

Let $K$ be a knot in a $\mathbb Z$-homology $3$-sphere with exterior $M_K$. The {\it longitude} $\lambda_K$ of $K$ is the longitude of $M_K$. The {\it meridian} $\mu_K$ of $K$ is the longitude of the tubluar neighbourhood $\overline{W \setminus M_K}$ of $K$. The pair $\mu_K, \lambda_K$ forms a basis for $H_1(\partial M_K)$. 

\begin{lemma} \label{horizontal}
Suppose that $T$ is a torus in a $\mathbb Z$-homology $3$-sphere $W$ and let $X, Y$ be the components of $W$ cut open along $T$. Suppose that $Y = P \cup Y_0$ where $P \cap Y_0 = \partial P \setminus T$ and $P$ is a Seifert manifold or than $S^1 \times D^2$ and $S^1 \times S^1 \times I$. Then 

$(1)$ the underlying space $B$ of the base orbifold of $P$ is planar, hyperbolic, and the multiplicities of the exceptional fibres in $P$ are pairwise coprime; 

$(2)$ $P$ has a unique Seifert structure;

$(3)$ if $\phi$ is the $P$-fibre slope on $T$ and $P$ has an exceptional fibre, then $\phi \not \in \{\lambda_X, \lambda_Y\}$.
\end{lemma}

\begin{proof}
If $B$ is non-orientable, or is orientable of positive genus, or has two exceptional fibres whose multiplicities are not coprime, then $W$ admits a degree $1$ map to a manifold with non-trivial first homology group, which is impossible. Thus (1) holds. Assertion (2) is a consequence of (1) and the classification of Seifert structures on $3$-manifolds (cf. \cite[\S VI.16]{Ja}).  Finally observe that as $H_1(Y(\lambda_X)) \cong \{0\}$ and $H_1(Y(\lambda_Y)) \cong \mathbb Z$, neither $Y(\lambda_X)$ nor $Y(\lambda_Y)$ has a lens space summand. On the other hand, if $P$ has an exceptional fibre, then $Y(\phi)$ does have such a summand. This completes the proof. 
\end{proof}

\section{Horizontal foliations and strongly detected slopes in Seifert fibred $\mathbb Z$-homology solid tori} \label{horizontal and detected} 

The set $\mathcal{S}_{rat}(T)$ of (rational) slopes on a torus $T$ is naturally identified with the subset $P(H_1(T; \mathbb Q))$ of the projective space $\mathcal{S}(T) = P(H_1(T; \mathbb R)) \cong S^1$. We endow $\mathcal{S}_{rat}(T)$ with the induced topology as a subset of $\mathcal{S}(T)$. The projective class of an element $\alpha \in H_1(T; \mathbb R)$ will be denoted by $[\alpha]$, though we sometimes abuse notation and write $\alpha \in \mathcal{S}_{rat}(T)$ for a non-zero class $\alpha$ in $H_1(T)$. 

For a $3$-manifold $X$ whose boundary is a torus $T$, set $\mathcal{S}_{rat}(X) = \mathcal{S}_{rat}(T)$. 
We say that $[\alpha] \in \mathcal{S}_{rat}(X)$ is {\it strongly detected} by a taut foliation $\mathcal{F}$ on $X$ if $\mathcal{F}$ restricts on $T$ to a fibration of slope $[\alpha]$. In this case we call $[\alpha]$ the {\it slope of $\mathcal{F}$}. 

When $X$ is Seifert fibred and $T$ is a boundary component of $X$, we say that $[\alpha] \in \mathcal{S}_{rat}(X)$ is {\it horizontal} if it is not the fibre slope. 

\begin{lemma} \label{strongly rational not vertical}
Suppose that $\mathcal{F}$ is a co-oriented taut foliation on a $\mathbb Z$-homology $3$-sphere $W$.

$(1)$ If $\mathcal{F} \cap T$ is a fibration by simple closed curves for some boundary component $T$ of a piece $P$ of $W$, then the slope of $T$ represented by these curves is horizontal.  

$(2)$ If $\mathcal{F}$ is rational, then it is horizontal. 
\end{lemma}

\begin{proof}
Suppose that $\mathcal{F} \cap T$ is a fibration by simple closed curves of vertical slope $\phi$ and let $P'$ be the manifold obtained by the $(T, \phi)$-Dehn filling $P$. Since $P$ has base orbifold of the form $B(a_1, \ldots, a_n)$ for a planar surface $B$ (Lemma \ref{horizontal}), $P'$ is homeomorphic to $(\#_{i = 1}^{n} L_{a_i}) \# (\#_{j= 1}^{r-1} S^1 \times D^2)$ where $r = |\partial P| - 1$. On the other hand, $\mathcal{F}$ extends to a co-oriented taut foliation $\mathcal{F}'$ on $P'$ and so $P'$ is either prime or $S^2 \times I$ (see e.g. \cite[Corollary 9.1.9]{CC2}). As the latter case does not arise, we have $n + (r-1) \leq 1$. Thus $P$ is either a solid torus or $S^1 \times S^1 \times I$, which is impossible for a piece of $W$. Thus part (1) the lemma holds. 

Next suppose that $\mathcal{F}$ is rational and let $P$ be a piece of $W$. By part (1), for each boundary component $T$ of $P$, $\mathcal{F} \cap T$ is a fibration by simple closed horizontal curves. Since the base orbifold of $P$ is planar (Lemma \ref{horizontal}), we can now argue as in the proof of \cite[Proposition 3]{Br1} to see that if $\mathcal{F}$ is not horizontal in $P$, it contains a vertical, separating leaf homeomorphic to a torus. This is impossible as it contradicts the assumption that $\mathcal{F}$ is co-oriented and taut (\cite[Proposition 2.1]{Go}). Thus part (2) holds. 
\end{proof}

Here is a special case of our main theorem.  

\begin{prop} \label{seifert case}
Let $W$ be a Seifert fibred $\mathbb Z$-homology $3$-sphere. Then the following conditions are equivalent:

$(a)$ $\pi_1(W)$ is left-orderable. 

$(b)$ $W$ is not an L-space.

$(c)$ $W$ admits a co-oriented horizontal foliation. 

Further, $W$ satisfies these conditions if and only if it is neither $S^3$ nor $\Sigma(2,3,5)$. 

\end{prop} 

\begin{proof}
Lemma \ref{horizontal} implies that the base orbifold $\mathcal{B}$ of $W$ has underlying space $S^2$. In this case the equivalence of (a) and (c) was established in \cite{BRW}, while those of (b) and (c) was established in \cite{LS} (see also \cite{CM}). 

Next suppose that $W$ is either $S^3$ or $\Sigma(2,3,5)$. Then the fundamental group of $W$ is finite so its fundamental group is not left-orderable, $W$ is an L-space \cite[Proposition 2.3]{OS4} and therefore it does not admit a co-oriented horizontal foliation \cite[Theorem 1.4]{OS1}. 

Conversely suppose that $W \ne S^3, \Sigma(2,3,5)$. Equivalently, $\chi(\mathcal{B}) \leq 0$. If $\chi(\mathcal{B}) = 0$, $\mathcal{B}$ would support a Euclidean structure and would therefore be one of $S^2(2,3,6), S^2(2,4,4), S^2(3,3,3)$ or $S^2(2,2,2,2)$. But then $H_1(\mathcal{B}) \ne \{0\}$ contrary to the fact that $H_1(W) = \{0\}$. Thus $\chi(\mathcal{B}) < 0$, so $\mathcal{B}$ is hyperbolic. It follows that there is a discrete faithful representation $\pi_1(\mathcal{B}) \to PSL_2(\mathbb R)$ and therefore a non-trivial homomorphism $\pi_1(W) \to PSL_2(\mathbb R)$. As $H^2(W) = \{0\}$, this homomorphism factors through $\widetilde{SL_2} \leq \widetilde{Homeo}_+(S^1) \leq \hbox{Homeo}_+(\mathbb R)$. Hence $\pi_1(W)$ is left-orderable (cf. \cite[Theorem 1.1(1)]{BRW}). It follows from the first paragraph of the proof that $W$ is not an L-space and it admits a co-oriented horizontal foliation. 
\end{proof}

Let $X$ be a Seifert fibered $\mathbb Z$-homology solid torus and set 
$$\mathcal{D}_{rat}^{str}(X) = \{ [\alpha] \in \mathcal{S}_{rat}(X) : [\alpha] \hbox{ is strongly detected by a rational foliation on $X$}\}$$
Clearly $\mathcal{D}_{rat}^{str}(X)$ coincides with the set of slopes $\alpha$ on $\partial X$ such that $X(\alpha)$ admits a horizontal foliation (cf. Lemma \ref{strongly rational not vertical}). The work of a number of people (\cite{EHN}, \cite{JN}, \cite{Na}) shows that the latter set is completely determined by the Seifert invariants of $X(\alpha)$.  In particular, we have the following result. 

\begin{prop} \label{strongly detected slopes}
Let $X$ be a Seifert manifold which is a $\mathbb Z$-homology solid torus with incompressible boundary. Then there is a connected open proper subset $U$ of $\mathcal{S}(X)$ such that 

$(1)$ $\mathcal{D}_{rat}^{str}(X) = U \cap \mathcal{S}_{rat}(X)$. 

$(2)$ If $X$ is not contained in $S^3$ and $\Sigma(2,3,5)$, then $U$ contains all the slopes $\alpha$ on $\partial X$ such that $X(\alpha)$ is a $\mathbb Z$-homology $3$-sphere.
\end{prop} 

\begin{proof}
The base orbifold of $X$ is of the form $D^2(a_1, a_2, \ldots , a_n)$ where $n$ and each $a_i$ are at least $2$. Since $X$ is a $\mathbb Z$-homology solid torus, the $a_i$ are pairwise coprime.  We can assume that the Seifert invariants $(a_1, b_1), \ldots (a_n, b_n)$ satisfy $0 < b_i < a_i$ for each $i$. Then 
$$\pi_1(X) = \langle y_1, y_2, \ldots , y_n, h : h \hbox{ central, } y_1^{a_1} = h^{b_1}, y_2^{a_2} = h^{b_2}, \ldots , y_n^{a_n} = h^{b_n} \rangle  $$
Further,    
$$h^* = y_1y_2\ldots y_n$$ 
is a peripheral element of $\pi_1(X)$ {\it dual} to $h$. That is, $H_1(\partial X) = \pi_1(\partial X)$ is generated by $h$ and $h^*$.  

Set $\gamma_i = \frac{b_i}{a_i}$. If $\alpha = ah + bh^*$ is a slope on $\partial X$, then $X(\alpha)$ has Seifert invariants $(0; 0; \gamma_1, \ldots, \gamma_n, \frac{a}{b})$ and therefore also $(0; -\lfloor \frac{a}{b}\rfloor; \gamma_1, \ldots, \gamma_n, \{\frac{a}{b}\})$ where $\{\frac{a}{b}\} = \frac{a}{b} - \lfloor \frac{a}{b}\rfloor$. 
According to \cite{EHN}, \cite{JN}, \cite{Na}, $X(\alpha)$ admits a horizontal foliation if and only if one of the following conditions holds:
\begin{enumerate}

\vspace{-.2cm} \item $1 - n < \frac{a}{b} < -1$;

\vspace{.2cm} \item $\lfloor \frac{a}{b}\rfloor = -1$ and there are coprime integers $0 < A < M$ and some permutation $(\frac{A_1}{M}, \frac{A_2}{M}, \ldots , \frac{A_{n+1}}{M})$ of $(\frac{A}{M}, \frac{M-A}{M}, \frac{1}{M}, \ldots , \frac{1}{M})$ such that 
$\gamma_i < \frac{A_i}{M} \hbox{ for } 1 \leq i \leq n$ and  $\{\frac{a}{b}\} < \frac{A_{n+1}}{M}$; 

\vspace{.2cm} \item $\lceil \frac{a}{b}\rceil = 1-n$ and there are coprime integers $0 < A < M$ and some permutation $(\frac{A_1}{M}, \frac{A_2}{M}, \ldots , \frac{A_{n+1}}{M})$ of $(\frac{A}{M}, \frac{M - A}{M}, \frac{M-1}{M}, \ldots , \frac{M-1}{M})$ such that 
$\gamma_i > \frac{A_i}{M} \hbox{ for } 1 \leq i \leq n$ and  $\{\frac{a}{b}\} > \frac{A_{n+1}}{M}$. 

\end{enumerate}
Let $V \subset \mathbb R$ be the convex hull of the set of rationals $\frac{a}{b}$ determined these three conditions. We leave it to the reader to verify that $V$ is an open interval if and only if $n > 2$ or $n = 2$ and $\gamma_1 + \gamma_2 \ne 1$ (cf. \cite[Proposition A.4]{BC}). On the other hand, our hypothesis that $X$ is a $\mathbb Z$-homology solid torus rules out the possibility that $n = 2$ and $\gamma_1 + \gamma_2 = 1$. Thus if $U$ is the connected proper subset of $\mathcal{S}(X)$ corresponding to $V$ under the identification $\frac{a}{b} \leftrightarrow [ah + bh^*]$, then $U$ is open and $\mathcal{D}_{rat}^{str}(X) = U \cap \mathcal{S}_{rat}(X)$, which proves (1). Part (2) then follows from Proposition \ref{seifert case}.
\end{proof}

The case when $X$ is contained in $S^3$ or $\Sigma(2,3,5)$ is dealt with in the following two propositions.

\begin{prop} \label{limit}
Let $X$ be a $(p, q)$ torus knot exterior where $p, q \geq 2$ and fix a meridian-longitude pair $\mu, \lambda$ for $X$ such that the Seifert fibre of $X$ has slope $pq \mu + \lambda$. Identify the non-meridional slopes on $\partial X$ with $\mathbb Q$ in the usual way: $m \mu + n \lambda \leftrightarrow \frac{m}{n}$. Then there is a co-oriented horizontal foliation of slope $r \in \mathbb Q$ in $X$ if and only if $r < pq - (p+q)$. In particular, the result holds for each $r < 1$. 
\end{prop}

\begin{proof}
Fix integers $a, b$ such that $1 = bp + aq$ and $0 < a < p$.  Note that $b < 0$ but $p(q + b) > aq + pb =1$, so $0 < b_0 = b + q < q$. There is a Seifert structure on $X$ with base orbifold $D^2(p,q)$ where the two exceptional fibres have Seifert invariants $(p, a)$ and $(q, b)$. Hence if $r = \frac{n}{m} \ne pq$ is a reduced rational fraction where $m > 0$, the Dehn filling $X(r)$ of $X$ is a Seifert fibred manifold with Seifert invariants $(0; 0; \frac{a}{p}, \frac{b}{q}, \frac{m}{n - mpq}) = (0; 0; \frac{a}{p}, \frac{b}{q}, \frac{1}{r - pq})$. Then $X(r)$  also has a Seifert structure with Seifert invariants $(0; 1 - \lfloor \frac{1}{pq - r} \rfloor; \frac{a'}{p}, -\frac{b}{q},  \{\frac{1}{pq - r}\})$ where $a' = p - a$. Assume that $ \{\frac{1}{pq - r}\} \ne 0$. Then arguing as in the proof of Proposition \ref{strongly detected slopes}, if $X(r)$ admits a horizontal foliation, we have $\lfloor \frac{1}{pq - r} \rfloor \in \{-1, 0\}$. If $\lfloor \frac{1}{pq - r} \rfloor = -1$, then $X(r)$ has Seifert invariants $(0; 1 ; \frac{a}{p}, \frac{b_0}{q},  1 - \{\frac{1}{pq - r}\})$ and there are positive integers $A_1, A_2$ coprime with an integer $M < A_1, A_2$ such that $\frac{a}{p} < \frac{A_1}{M}, \frac{b_0}{q} < \frac{A_2}{M}$ and $\frac{A_1 + A_2}{M} \leq 1$. But this is impossible since then $\frac{A_1 + A_2}{M} > \frac{a}{p} + \frac{b_0}{q} = 1 + \frac{1}{pq}$. Hence $\lfloor \frac{1}{pq - r} \rfloor = 0$ and therefore $0 < \frac{1}{pq - r} < 1$ and $X(r)$ has Seifert invariants $(0; 1 ; \frac{a'}{p}, -\frac{b}{q},  \{\frac{1}{pq - r}\})$. It follows that $r < pq -1$. A straightforward, though tedious, calculation yields the bound stated in the proposition. This calculation can be avoided if we are willing to appeal to results from Heegaard-Floer theory. For instance, the $(p, q)$ torus knot  $K$ is an L-space knot since $pq-1$ surgery on $K$ yields a lens space. Hence as the genus of $K$ is $\frac{1}{2} (p-1)(q-1)$, $K(r)$ is an L-space if and only if $r \geq pq - (p+q)$ (\cite[Proposition 9.5]{OS5}. See also \cite[Fact 2, page 221]{Hom}). Hence, according to Proposition \ref{seifert case}, $X(r)$ admits a horizontal foliation if and only if $r < pq - (p+q)$.
\end{proof}

\begin{prop} \label{limit235}
Let $X$ be a Seifert manifold which is the exterior of a knot $K$ in $\Sigma(2,3,5)$, the Poincar\'e homology $3$-sphere. 

$(1)$ $K$ is a fibre in a Seifert structure on $\Sigma(2,3,5)$. 

$(2)$ $X$ has base orbifold $D^2(2,3), D^2(2, 5), D^2(3,5)$, or $D^2(2,3,5)$.

$(3)$ Suppose that $K$ has multiplicity $j \geq1$. Then there is a choice of meridian $\mu$ and longitude $\lambda$ of $K$ such that $X$ admits a horizontal foliation detecting the slope $a \mu + b \lambda$ if and only if 
$$\frac{a}{b} > -29 \hbox{ if } j = 1$$ 
and 
$$\frac{a}{b} <  \left\{ \begin{array}{ll} 
7 & \hbox{ if } j = 2 \\
3 & \hbox{ if } j = 3 \\
1 & \hbox{ if } j = 5 
\end{array} \right. $$
In particular, there is a sequence of slopes $\alpha_n$ on $\partial X$ which converge projectively to the meridian of $K$ such that $X$ admits a horizontal foliation of slope $\alpha_n$ for each $n$. 

$(4)$ There is a unique slope on $\partial X$ such that $X(\alpha) \cong \Sigma(2,3,5)$. 

\end{prop} 

\begin{proof}
The boundary of $X$ is incompressible since the fundamental group of $\Sigma(2,3,5)$ is non-abelian. It follows from Lemma \ref{horizontal} that $X$ has base orbifold of the form $D^2(a_1, a_2, \ldots, a_n)$ where each $a_i \geq 2$ and $n \geq 2$. Since $\Sigma(2,3,5)$ has no lens space summands, the meridian of $K$ cannot be the fibre slope of $X$. Thus the Seifert structure on $X$ extends to one on $\Sigma(2,3,5)$ in which $K$ is a fibre. This implies assertions (1) and (2) of the proposition.

Next we deal with (3). 
Let $K_j$ be a fibre of multiplicity $j$ in $\Sigma(2,3,5)$ for $j = 1, 2, 3, 5$ and let $X_0$ be the exterior of $K_1 \cup K_2 \cup K_3 \cup K_5$. Denote by $T_j$ the boundary component of $X_0$ corresponding to $K_j$ and by $\mu_j$ the meridional slope of $K_j$ on $T_j$. Let $\phi_j$ be the fibre slope on $T_j$. Note that $X_0$ is a trivial circle bundle over a $4$-punctured sphere $Q$. Orient $Q$. Since $\Sigma(2,3,5)$ has Seifert invariants $(0; -1, \frac12, \frac13, \frac15)$, there is a section of this bundle with image $\widetilde Q \subset X_0$ such that if $\sigma_j$ is the slope of $\widetilde Q \cap T_j$ oriented by the induced orientation from $Q$. Orient the fibre of $X_0$ so that for each $j$, $\sigma_j \cdot \phi_j = 1$. 

There is a horizontal foliation on $X_j$ detecting the slope $n \sigma_j + m \phi_j$ if and only if the ($n \sigma_j + m \phi_j$)-Dehn filling of $X_j$ admits a horizontal foliation. The latter problem has been resolved in the papers \cite{EHN}, \cite{JN}, and \cite{Na}. First we prove that $X_j$ has a horizontal foliation if and only if $\frac{m}{n} \in (-1, 0)$ for $j =1$ and $\frac{m}{n} \in (0, \frac{1}{j})$ for $j > 1$.  

The exterior $X_j$ of $K_j$ is obtained from $X_0$ by performing the $(T_k, \mu_k)$-filling for $k \ne j$. It follows that the $(n \sigma_j + m \phi_j)$-Dehn filling of $X_j$ has Seifert invariants 
\begin{itemize}

\vspace{-.2cm} \item $(0; -1, \frac12, \frac13, \frac15,  \frac{m}{n})$ if $j = 1$;

\vspace{.2cm} \item $(0; -1, \frac13, \frac15, \frac{m}{n})$ if $j = 2$;

\vspace{.2cm} \item $(0; -1, \frac12, \frac15, \frac{m}{n})$ if $j = 3$;

\vspace{.2cm} \item $(0; -1, \frac12, \frac13, \frac{m}{n})$ if $j = 5$.
 
\end{itemize}
\vspace{-.2cm}
Suppose first that $j = 1$. If $n = 0$, $X_1(n \sigma_1+ m \phi_1) = X_1(\phi_1)$ is a connected sum of lens spaces of orders $2, 3,$ and $5$ so does not admit a taut foliation (see e.g. \cite[Corollary 9.1.9]{CC2}). If $|n| = 1$, then $\Delta (n \sigma_1+ m \phi_1, \phi_1) = 1$, so $X_1(n \sigma_1+ m \phi_1)$ admits a Seifert structure with base orbifold $S^2(2,3,5)$. Hence it has a finite fundamental group and so does not admit a horizontal foliation. Assume then that $|n| > 1$, and therefore $0 < \{\frac{m}{n}\} = \frac{m}{n} - \lfloor \frac{m}{n}\rfloor < 1$. In this case, $X_1(n \sigma_1 + m \phi_1)$ has Seifert invariants $(0; \lfloor \frac{m}{n}\rfloor - 1, \frac12, \frac13, \frac15, \{\frac{m}{n}\})$. Theorem 2 of \cite{JN} implies that when $\lfloor \frac{m}{n}\rfloor = -1$ there is a horizontal foliation for all values of $\{\frac{m}{n}\}$. In other words, whenever $\frac{m}{n} \in (-1, 0)$. It also shows that there is no horizontal foliation when $\lfloor \frac{m}{n}\rfloor < -2$ or $\lfloor \frac{m}{n}\rfloor > 0$

If $\lfloor \frac{m}{n}\rfloor = 0$, then $X_1(n \sigma_1 + m \phi_1)$ has Seifert invariants $(0; - 1, \frac12, \frac13, \frac15, \{\frac{m}{n}\})$. Conjecture 2 of \cite{JN} was verified in \cite{Na} so in this case $X_1(n \sigma_1 + m \phi_1)$ has a horizontal foliation if and only if we can find coprime integers $0 < A < M$ such that for some permutation $\{\frac{a_1}{m_1}, \frac{a_2}{m_2}, \frac{a_3}{m_3}, \frac{a_4}{m_4}\}$ of $\{\frac12, \frac13, \frac15, \{\frac{m}{n}\}\}$ satisfies $\frac{a_1}{m_1} < \frac{1}{M}, \frac{a_2}{m_2} < \frac{1}{M}, \frac{a_3}{m_3} < \frac{A}{M}$ and $\frac{a_4}{m_4} < \frac{M-A}{M}$. It is elementary to verify that there is no such pair $A, M$. 

If $\lfloor \frac{m}{n}\rfloor = -2$, then $X_1(n \sigma_1 + m \phi_1)$ has Seifert invariants $(0; - 3, \frac12, \frac13, \frac15, \{\frac{m}{n}\})$ and therefore also $(0; - 1, \frac12, \frac23, \frac45, 1 - \{\frac{m}{n}\})$. As in the previous paragraph, $X_1(n \sigma_1 + m \phi_1)$ never admits a horizontal foliation on this case. We conclude that $X_1(n \sigma_1 + m \phi_1)$ admits a horizontal foliation if and only if $\frac{m}{n} \in (-1, 0)$. 

We proceed similarly when $j = 2$. As above we can rule out the cases $n = 0$ and $|n| = 1$. When $|n| > 1$, so $0 < \{\frac{m}{n}\} = \frac{m}{n} - \lfloor \frac{m}{n}\rfloor < 1$, $X_2(n \sigma_2 + m \phi_2)$ has Seifert invariants $(0; \lfloor \frac{m}{n}\rfloor - 1, \frac13, \frac15, \{\frac{m}{n}\})$. By Theorem 2 of \cite{JN}, there is no horizontal foliation when $\lfloor \frac{m}{n}\rfloor < -1$ or $\lfloor \frac{m}{n}\rfloor > 0$. If $\lfloor \frac{m}{n}\rfloor = 0$, $X_2(n \sigma_2 + m \phi_2)$ has Seifert invariants $(0; - 1, \frac13, \frac15, \{\frac{m}{n}\})$. Conjecture 2 of \cite{JN} was verified in \cite{Na} so in this case $X_2(n \sigma_2 + m \phi_2)$ has a horizontal foliation if and only if we can find coprime integers $0 < A < M$ such that for some permutation $\{\frac{a_1}{m_1}, \frac{a_2}{m_2}, \frac{a_3}{m_3}\}$ of $\{\frac13, \frac15, \{\frac{m}{n}\}\}$ satisfies $\frac{a_1}{m_1} < \frac{1}{M}, \frac{a_2}{m_2} < \frac{A}{M}$ and $\frac{a_3}{m_3} < \frac{M-A}{M}$. It is elementary to verify that there is a solution to this problem if and only if $\frac{m}{n} \in (0, \frac12)$. 
On the other hand, if $\lfloor \frac{m}{n}\rfloor = -1$, $X_1(n \sigma_1 + m \phi_1)$ has Seifert invariants $(0; -2, \frac13, \frac15, \{\frac{m}{n}\})$ and therefore $(0; -1, \frac23, \frac45, 1-\{\frac{m}{n}\})$. As above, $X_2(n \sigma_2 + m \phi_2)$ never admits a horizontal foliation on this case. We conclude that $X_2(n \sigma_2 + m \phi_2)$ admits a horizontal foliation if and only if $\frac{m}{n} \in (0, \frac12)$. 

We leave the cases $j = 3, 5$ to the reader. 

To complete the proof of (3) we must express the conclusions we have just obtained in terms of appropriately chosen meridians and longitudes for the knots $K_j$. We proceed as follows. The euler number of $X_j(n \sigma_j + m \phi_j)$ is given, up to sign, by the sum of its Seifert invariants. Further, since $H_1(X_j(\lambda_j)) \cong \mathbb Z$, we can solve for the coefficients $n, m$ of $\lambda_j$. For instance for $j > 1$, set $\{j, p, q\} = \{2,3,5\}$. If $\lambda_j = n\sigma_j + m\phi_j$, then $0 = |e(X_j(n\sigma_j + m\phi_j))| = |-1 + \frac{1}{p} + \frac{1}{q} + \frac{m}{n}|$. Thus $\frac{m}{n} = \frac{pq - (p+q)}{pq}$. Since $\gcd(pq, pq - (p+q)) = 1$, we have 
$$\lambda_j = -pq \sigma_j + (p + q - pq ) \phi_j$$ 
Similarly for $j = 1$ we have $\frac{m}{n} = 1 - (\frac12 + \frac13 + \frac15) = -\frac{1}{30}$. Hence 
$$\lambda_1 = -30 \sigma_1 + \phi_1$$ 
The $\mu_j$ Dehn filling of $X_j$ yields $\Sigma(2,3,5)$ and it is known that $|e(\Sigma(2,3,5))| = \frac{1}{30}$. Combined with the identity $\Delta(\mu_j, \lambda_j) = 1$ we can solve for the coefficients of $\mu_j$: 
$$\mu_j = \left\{ \begin{array}{ll} 
\sigma_1 & \hbox{if } j = 1 \\ 
j\sigma_j + \phi_j & \hbox{if } j > 1 
\end{array} \right.$$
With these choices, it is easy to verify that the set of detected slopes $a \mu_1 + b \lambda_1$ corresponds to the interval specified in (3). 

To prove (4), let $\alpha = a \mu_j + b \lambda_j$ be a slope on $\partial X_j$ such that $X_j(\alpha) \cong \Sigma(2,3,5)$. Since $\Sigma(2,3,5)$ is a $\mathbb Z$-homology $3$-sphere, $1 = \Delta(\alpha, \lambda_j) = |a|$. Without loss of generality we can suppose that $a = 1$. On the other hand, the core of the filling torus in $X_j(\alpha)$ is $K_j$, so
$$j = \Delta(\alpha, \phi_j) = \left\{ 
\begin{array}{ll} 
\Delta( \mu_j + b \lambda_j, 30 \mu_1 + \lambda_1) & \hbox{ if } j = 1 \\
\Delta(\mu_j + b \lambda_j, pq \mu_j + j \lambda_j) & \hbox{ if } j > 1
\end{array} \right.$$
$$ \hspace{-.4cm} = \left\{ 
\begin{array}{ll} 
|1 - 30b| & \hbox{ if } j = 1 \\
|j - pqb| & \hbox{ if } j > 1
\end{array} \right.$$
Hence there is an $\epsilon \in \{\pm 1\}$ such that $j \epsilon = \left\{ 
\begin{array}{ll} 
1 - 30b & \hbox{ if } j = 1 \\
j - pq b & \hbox{ if } j > 1
\end{array} \right.$. It follows that $b = 0$ so that $\alpha = \mu_j$. This proves (4).
\end{proof}

\begin{cor} \label{s3 and s235}
Suppose that $K$ is a knot in either $S^3$ or $\Sigma(2,3,5)$ whose exterior $X$ is Seifert fibered and let $U$ be the connected open subset of $\mathcal{S}(X)$ described in Proposition \ref{strongly detected slopes}. 

$(1)$ If $X$ is the trefoil exterior, then $U$ contains all the slopes $\alpha$ on $\partial X$ such that $X(\alpha)$ is a $\mathbb Z$-homology $3$-sphere other than $S^3$ and $\Sigma(2,3,5)$. The two slopes yielding the latter two manifolds are the end-points of $\overline{U}$.

$(2)$ If $X$ is not the trefoil exterior, then $U$ contains all the slopes $\alpha$ on $\partial X$ such that $X(\alpha)$ is a $\mathbb Z$-homology $3$-sphere other than the meridian of $K$, which is an end-point of $\overline{U}$.
\qed
\end{cor}

\section{Existence of rational foliations on aspherical graph  $\mathbb Z$-homology $3$-spheres} \label{existence}

We prove Theorem \ref{taut} in this section by induction on the number of its JSJ pieces, the base case being dealt with in Proposition \ref{seifert case}. We suppose below that $W$ is a non-Seifert graph manifold $\mathbb Z$-homology $3$-sphere. 

\begin{lemma} \label{distance 1}
Suppose that $M$ is a graph manifold $\mathbb Z$-homology solid torus with incompressible boundary. If $\alpha$ and $\beta$ are slopes on $\partial M$ whose associated fillings are $\mathbb Z$-homology $3$-spheres which are either $S^3, \Sigma(2,3,5)$ or reducible, then $\Delta(\alpha, \beta) \leq 1$. 
\end{lemma}

\begin{proof}
If $M$ is Seifert fibred, it has base orbifold $D^2(a_1, \ldots, a_n)$ where $n$ and each $a_i$ are at least $2$. Further, the $a_i$ are pairwise coprime. In this case $M$ admits no fillings which are simultaneously reducible and $\mathbb Z$-homology $3$-spheres. Thus $M(\alpha)$ and $M(\beta)$ are either $S^3$ or $\Sigma(2,3,5)$. If $\alpha$ and $\beta$ are distinct slopes, then $M(\alpha)$ and $M(\beta)$ cannot both be $S^3$ as torus knots admit unique $S^3$-surgery slopes. Similarly Proposition \ref{limit235} implies that $M(\alpha)$ and $M(\beta)$ cannot both be $\Sigma(2,3,5)$. On the other hand, if one of $M(\alpha)$ and $M(\beta)$ is $S^3$ and the other $\Sigma(2,3,5)$, then $M$ must be the trefoil knot exterior and $\Delta(\alpha, \beta) = 1$.

Next suppose that $M$ is not Seifert fibred. If $M(\alpha)$ is reducible, then the main result of \cite{GLu} combines with \cite[Theorem 1.2]{BZ2} to show that $\Delta(\alpha, \beta) \leq 1$. On the other hand, if $M(\alpha)$ and $M(\beta)$ are either $S^3$ or $\Sigma(2,3,5)$ and $\Delta(\alpha, \beta) \geq 2$, then \cite[Theorem 1.2(1)]{BZ1} implies that $M$ has two pieces, one a cable space and the other a Seifert manifold $M_0$ with base orbifold a $2$-disk with two cone points. The proof of \cite[Theorem 1.2(1)]{BZ1} (see \S 8 of \cite{BZ1}) now implies that $M_0$ admits two Dehn fillings yielding $S^3$ or $\Sigma(2,3,5)$ whose slopes are of distance at least $8$, which is impossible. (See the discussion which follows the statement of \cite[Theorem 1.2]{BZ1}.) Thus $\Delta(\alpha, \beta) \leq 1$.
\end{proof}

Let $X$ be a piece of $W$ whose boundary is a torus. (Thus $X$ corresponds to a leaf of the JSJ-graph of $W$.) If $Y = \overline{W \setminus X}$ is the exterior of $X$ in $W$, then $T = X \cap Y$ is an essential torus. Let $\lambda_X$ and $\lambda_Y$ be the longitudes of $X$ and $Y$. For slopes $\alpha$ and $\beta$ on $T$ we have 
$$|H_1(X(\alpha))| = \Delta(\alpha, \lambda_X) \;\;\; \hbox{and} \;\;\; |H_1(Y(\beta))| = \Delta(\beta, \lambda_Y)$$
Hence as we noted in \S \ref{sec: detecting slopes} that $\Delta(\lambda_X, \lambda_Y) = 1$, both $X(\lambda_Y)$ and $Y(\lambda_Y)$ are $\mathbb Z$-homology $3$-spheres. 

Let $\phi_X$ and $\phi_Y$ be primitive elements of $H_1(T)$ representing, respectively, the slopes of the Seifert fibre of $X$ and that of the piece $P$ of $Y$ incident to $T$. Since $X$ has exceptional fibres, $\pm \phi_X \not \in \{\lambda_X, \lambda_Y\}$ (Lemma \ref{horizontal}(3)). It follows that $X(\lambda_X)$ and $X(\lambda_Y)$ are irreducible Seifert manifolds (Lemma \ref{horizontal}(1)). 

\begin{proof}[Proof of Theorem \ref{taut}]
For an integer $n$, set 
$$\alpha_n = \lambda_X + n \lambda_Y$$ 
and observe that $\lim_{|n|} [\alpha_n] = [\lambda_Y] \in \mathcal{S}_{rat}(T)$. Since $X(\lambda_Y)$ is a $\mathbb Z$-homology $3$-sphere, $\alpha_n$ is strongly detected by a horizontal foliation in $X$ for $n \gg 0$ or for $n \ll 0$ or for both (Proposition \ref{strongly detected slopes} and Corollary \ref{s3 and s235}). To complete the proof it suffices to find a rational foliation of $Y$ which strongly detects $\alpha_n$ for all large $|n|$.

Since $\Delta(\alpha_n, \lambda_Y) = 1$, the manifolds $Y(\alpha_n)$ are $\mathbb Z$-homology 3-spheres, and since $Y$ is irreducible and $\Delta(\alpha_n, \alpha_m) = |n-m|$, there are at most two $n$ such that $Y(\alpha_n)$ is either reducible, $S^3$ or $\Sigma(2,3,5)$, and if two, they are successive integers (Lemma \ref{distance 1}). Thus for $|n|$ large, $Y(\alpha_n)$ is an irreducible graph manifold $\mathbb Z$-homology $3$-sphere which is neither $S^3$ nor $\Sigma(2,3,5)$. Hence our inductive hypothesis implies that $Y(\alpha_n)$ admits a rational foliation $\mathcal{F}_n$ for large $|n|$. If $\lambda_Y \ne \phi_Y$, then as $\Delta(\alpha_n, \phi_Y) = |\alpha_n \cdot \phi_Y | \geq |n||\lambda_Y \cdot \phi_Y| - |\lambda_X \cdot \phi_Y|$, for large $|n|$ the JSJ pieces of $Y(\alpha_n)$ are $P(\alpha_n)$ and the JSJ pieces of $\overline{Y \setminus P}$. Thus $\mathcal{F}_n$ induces a rational foliation of slope $\alpha_n$ on $Y$, which completes the proof. 

Suppose then that $\lambda_Y = \phi_Y$. Then Lemma \ref{horizontal}(3) implies that $P$ is a product $F \times S^1$ where $F$ is a planar surface with $|\partial P| \geq 3$ boundary components. Since $\Delta(\alpha_n, \phi_Y) = \Delta(\alpha_n, \lambda_Y) = 1$, each $P(\alpha_n)$ is a product $\bar F \times S^1$ where $\bar F$ is a planar surface with $|\partial P|-1 \geq 2$ boundary components. If $|\partial P| \geq 4$, the JSJ pieces of $Y(\alpha_n)$ are $P(\alpha_n)$ and the JSJ pieces of $\overline{Y \setminus P}$, so we can proceed as above. 

Finally assume that $|\partial P| = 3$ and let $Y_1, Y_2$ be the components of $\overline{Y \setminus P}$. Denote the JSJ torus $Y_i \cap P$ by $T_i$, so $\partial P = \partial Y \cup T_1 \cup T_2$. For each $n$ we have $P(\alpha_n) \cong S^1 \times S^1 \times I$, so $Y(\alpha_n) \cong Y_1 \cup Y_2 \not \cong S^3, \Sigma(2,3,5)$. By induction, there is a rational foliation $\mathcal{F}_n$ on $Y(\alpha_n)$. Since there is no vertical annulus in $P$ which is cobounded by the Seifert fibres of the two pieces of $Y$ incident to $P$, the reader will verify that there is at most one value of $n$ for which there is an annulus in $P(\alpha_n)$ cobounded by these fibres. Thus for $|n| \gg 0$, $Y(\alpha_n)$ is a graph manifold $\mathbb Z$-homology 3-sphere whose pieces are the JSJ pieces of $\overline{Y \setminus P}$. Fix such an $n$ and note that up to isotopy, we can suppose that $\mathcal{F}_n$ is a product fibration on $P(\alpha_n) \cong S^1 \times S^1 \times I$ whose fibre is an annulus. It follows that we can choose primitive classes $\beta_n^1 \in H_1(T_1)$ and $\beta_n^2 \in H_1(T_2)$ representing the slopes of $\mathcal{F}_n$ on $T_1, T_2$ and an integer $k$ such that $k \alpha_n + \beta_n^1 + \beta_n^2 = 0$  in $H_1(P)$. 

Let $p: P = F \times S^1 \to F$ be the projection and denote by $a, b_1, b_2 \in H_1(F)$ the classes associated to the boundary components of $F$, where $a$ corresponds to $p(T)$, $b_1$ to $p(T_1)$, and $b_2$ to $p(T_2)$. We may assume that $a + b_1 + b_2 = 0$. Since $\Delta(\alpha_n, \phi_Y) = 1$, we can also assume that the projection $p: P \to F$ sends $\alpha_n$ to $a$. Fix integers $k_1, k_2$ so that $p_*(\beta_n^j) = k_j b_j$. Clearly $|k_j| = \Delta(\beta_n^j, \phi_j)$ where $\phi_j$ is the slope on $T_j$ determined by the Seifert structure on $P$. Then we have 
$$0 = p_*(k \alpha_n + \beta_n^1 + \beta_n^2) = ka + k_1 b_1 + k_2 b_2$$
in $H_1(F)$. This can only happen if $k = k_1 = k_2$. Thus if $k \ne 0$, the fibration in $P(\alpha_n)$ determined by $\mathcal{F}_n$ is horizontal in $P$ and of slope $\alpha_n$ on $T$, so we are done. 

Suppose then that $k = 0$, so $0 = |k_j| = \Delta(\beta_n^j, \phi_j)$. Thus $[\beta_n^1] = [\phi_1]$ and $[\beta_n^2] = [\phi_2]$ are vertical in $P$. By construction, $Y(\lambda_Y) = Y(\phi_Y) = Y_1(\phi_1) \# Y_2(\phi_2) = Y_1(\beta_n^1) \# Y_2(\beta_n^2)$ and as $\mathbb Z \cong  H_1(Y(\lambda_Y)) =  H_1(Y_1(\phi_1)) \oplus H_1(Y_2(\phi_2))$, we can suppose that $H_1(Y_1(\beta_n^1)) \cong \mathbb Z$ and $H_1(Y_2(\beta_n^2)) \cong \{0\}$. Thus $\phi_1 = \beta_n^1 = \lambda_{Y_1}$ and $\Delta(\phi_2, \lambda_{Y_2}) = \Delta(\beta_n^2, \lambda_{Y_2}) = 1$. 

Fix $\delta_0 \in H_1(T_1)$ such that $1 = \Delta(\delta_0, \lambda_{Y_1}) = \Delta(\delta_0, \phi_1)$ and $p_*(\delta_0) = b_1$. Then $p_*(\lambda_X + \delta_0 + 
 \lambda_{Y_2}) = a + b_1 + b_2 = 0 \in H_1(F)$ and therefore $\lambda_X + \delta_0 +  \lambda_{Y_2}  = j \phi_Y \in H_1(P)$ for some integer $j$. After replacing $\delta_0$ by $\delta_0 - j \phi_1$ we can suppose that 
$$\lambda_X + \delta_0 +  \lambda_{Y_2} = 0 \in H_1(P)$$
With this choice, set $\delta_m = \delta_0 + m \phi_1$.

\begin{claim}\label{horizontal1}  For all but at most finitely many $m$, $Y_1$ admits a rational foliation of slope $\delta_m$.
\end{claim}

\begin{proof}
Since $\Delta(\delta_m, \lambda_{Y_1}) = 1$ for all $m$, $Y_1(\delta_m)$ is a $\mathbb Z$-homology $3$-sphere. Let $\phi_{Y_1}$ be the primitive element of $H_1(T_1)$ representing the slope of the Seifert fibre of the piece $P_1$ of $Y_1$ incident to $T_1 = \partial Y_1$, then $\Delta(\lambda_{Y_1} ,\phi_{Y_1}) \geq 1$, since $\lambda_{Y_1} = \phi_{Y}$ and $T_1$ is a JSJ-torus of $Y$. Therefore our inductive hypothesis combines with Lemma \ref{distance 1} to show, as in the first part of the proof, that for all but at most fnitely many $m$, $Y_1$ admits a rational foliation of slope $\delta_m$.
\end{proof}

\begin{claim}\label{horizontal2}  $Y_2$ admits a rational foliation of slope $\gamma = p\lambda_{Y_2} + q\phi_2$ where $p$ and $q$ are relatively prime and non-zero.
\end{claim}

\begin{proof}
Let $\phi_{Y_2}$ be the primitive element of $H_1(T_2)$ representing the slope of the Seifert fibre of the piece $P_2$ of $Y_2$ incident to $T_2 = \partial Y_2$.
If $\Delta(\lambda_{Y_2} ,\phi_{Y_2}) \geq 1$, the assertion follows from the proof of Claim \ref{horizontal1} by taking $\gamma = p\lambda_{Y_2} +\phi_2$, for some $\vert p \vert$ sufficiently large.

We consider now the case where $\lambda_{Y_2} = \phi_{Y_2}$. Let $E \subset S^3$ be the trefoil exterior, $\mu_E \in H_1(\partial E)$ its meridional slope and $\nu_E \in H_1(\partial E)$ the unique slope such that $E(\nu_E) \cong \Sigma(2,3,5)$. Then $\Delta(\mu_E, \nu_E) = 1$. Further, $E$ does not admit a horizontal foliation of slope $\mu_E$ or $\nu_E$. We build a $\mathbb Z$-homology $3$-sphere $W_2 = E \cup Y_2$ by gluing $E$ and $Y_2$ along their boundaries in such a way that the slope $\mu_E$ is identified with the slope $\lambda_{Y_2} $ and the slope $\nu_E$ is identified with the slope $\phi_2$.  Since the fiber slope $ \phi_{Y_2} = \lambda_{Y_2}$ is identified with the meridional slope $\mu_E$, the Seifert fibrations on $E$ and $P_2$ do not match up, and the torus $\partial Y_2 = \partial E$ is a  JSJ-torus of $W_2$. Hence $W_2$ is a graph $\mathbb Z$-homology $3$-sphere whose JSJ pieces are $E$ and the $JSJ$ pieces of $Y_2$. In particular, $W_2$ has fewer pieces than $W$. By the inductive hypothesis $W_2$ carries a rational foliation which intersects the JSJ torus $\partial Y_2 = \partial E$ in a circle fibration of some slope $\gamma$. Hence $Y_2$ admits a rational foliation of slope $\gamma$. Moreover $\Delta(\gamma,  \lambda_{Y_2}) \geq 1$ and $\Delta(\gamma,  \phi_2) \geq 1$ since $E$ cannot admit a horizontal foliation of slope $\mu_E$ or $\nu_E$.
 \end{proof}
Now we complete the proof of Theorem \ref{taut}. 

For $\vert m \vert$ sufficiently large, let $\delta_m = \delta_0 + m \phi_1 \in H_1(T_1)$ be the slope of a rational foliation on $Y_1$ given by Claim \ref{horizontal1}, and  $\gamma = p\lambda_{Y_2} + q\phi_2 \in H_1(T_2)$ the slope of a rational foliation on $Y_2$ given by Claim \ref{horizontal2}. Since $\lambda_Y = \phi_Y = \phi_1 = \phi_2$ and 
$\lambda_X + \delta_0 +  \lambda_{Y_2} = 0$ in $H_1(P)$, the sum $ \zeta_m + p\delta_m + \gamma = 0 \in H_1(P)$ where $\zeta_m = p\lambda_X - (pm + q)\lambda_Y \in H_1(T)$ is a primitive class. Thus there is a properly embedded, horizontal surface $F_m$ in $P$ with boundary curves of slope $\zeta_m, \delta_m$ and $\gamma$. Hence  $P$ fibres over the circle with fibre $F_m$ and $Y$ admits a rational foliation of slope $\zeta_m$ for large $|m|$. Now, it is easy to verify that $\lim_{|m|} [\zeta_m] = [\lambda_Y]$ and that for large $|m|$, reversing the sign of $m$ sends $[\zeta_m] $ from one side of $[\lambda_Y]$ to the other. Since $X(\lambda_Y)$ is a $\mathbb Z$-homology $3$-sphere, Proposition \ref{strongly detected slopes} and Corollary \ref{s3 and s235} imply that $X$ admits a horizontal foliation of slope $\delta_m$ for $m \gg 0$ or for $m \ll 0$ or for both. This completes the induction and the proof of Theorem \ref{taut}.
\end{proof}

\end{document}